\documentclass{amsart}
\usepackage[mathscr]{eucal}
\usepackage{amssymb, amsmath,array, amscd}
\usepackage{enumerate}
\usepackage{graphicx}
\usepackage{url}
\usepackage[colorlinks,plainpages,backref]{hyperref}

\input xy
\xyoption{all}

\newtheorem{thm}{Theorem}[section]

\newtheorem{prop}{Proposition}[section]
\newtheorem{lemma}{Lemma}[section]
\newtheorem{definition}{Definition}[section]
\newtheorem{remark}{Remark}[section]

\newtheorem{example}{Example}[section]

\newcommand{\A}{{\mathcal A}}
\newcommand{\X}{{\mathcal X}}

\newcommand{\C}{{\mathbb C}}

\begin{document}

\title{Multinets in $\mathbb P^2$}

%
%
\footnote{Updated: October 24, 2013}

\author{J. Bartz}
\address{Department of Mathematics \\ Francis Marion University \\
Florence \\ SC 29506 USA} \email{ jbartz@fmarion.edu}

\author{S. Yuzvinsky}
\address{Department of Mathematics \\ University of Oregon \\ Eugene \\ OR 94703 USA}
\email { yuz@uoregon.edu }

\dedicatory{}
\thanks{}


\begin{abstract}
Multinets are certain configurations of lines and points with multiplicities in the complex projective plane $\mathbb{P}^2$.
They are used in the studies of resonance and characteristic varieties of complex hyperplane arrangement complements and cohomology of Milnor fibers.  From combinatorics viewpoint  they can be considered as generalizations of Latin squares. Very few examples
of multinets with non-trivial multiplicities are known. In this paper, we present new examples of multinets. These are obtained by using an analogue of  nets in $\mathbb{P}^3$
and intersecting them by planes. 

\end{abstract}

\maketitle

\date{\today}

\section{Introduction}
\label{S:1}

Multinets are certain configurations of lines and points with multiplicities in the complex projective plane $\mathbb{P}^2$. More exactly they are multi-arrangements of projective lines partitioned in three blocks with some extra properties (see section 2). They appeared in \cite{FY, LY} in the study of resonance and characteristic varieties of the complement of a complex hyperplane arrangement.  More recently, multinets have been used to study the cohomology of Milnor fibers such as in \cite{DS}.
Although multinets can be defined purely combinatorially using an incidence relation, very few examples
of multinets with non-trivial multiplicities are known. In the paper, we recall some definitions, describe a new method to obtain multinets and give quite a few new examples. For that we consider an analogue of nets in $\mathbb{P}^3$
and intersect them by planes.  

The paper is organized as follows. In section 2, we recall basic definitions and properties of multinets. In section 3, we give the general idea of constructing multinets. Section 4 contains the main part of the paper. We systematically go over
different cases of constructed multinets. The main parameters being numbers of lines and points of various multiplicities. All cases are classified except multinets with all lines having multiplicity 1 and points having multiplicities 1 and 2. For that case we have only examples and a uniform upper bound on the number of points of multiplicity 2. In section 5,
we discuss briefly the combinatorics inside blocks. The conclusion is that this combinatorics is defined by multinet structure, i.e., the combinatorics between blocks. Finally some open questions and conjectures are collected in section 6.

\section{Preliminaries}
\label{S:2}

\subsection{Pencils of curves and multinets in $\mathbb{P}^2$}

There are several equivalent ways to define multinets in $\mathbb{P}^2$. We introduce them  here using pencils of plane curves.
A {\it pencil of plane curves} is a line in the projective space 
of homogeneous polynomials from $\mathbb{C}[x_1,x_2,x_3]$ of some fixed degree $d$.  Any two distinct curves of the same degree generate a pencil, and conversely a pencil is 
determined by any two of its curves $C_1, C_2$. An arbitrary curve $C$ in the 
pencil (called a {\it fiber}) is $C=aC_1 + bC_2,\ [a:b]\in \mathbb P^1.$  Every two fibers in a pencil intersect in the same set of points $\X=C_1\cap C_2,$ called the {\it base} of the pencil. If fibers do not have a common component
called a ({\it fixed component}), then the base is a finite set of points.

A curve of the form  $\prod_{i=1}^q \alpha_i^{m_i},$ where $\alpha_i$ 
are distinct linear forms and $m_i\in \mathbb{Z}_{>0}$
for $1\leq i\leq q,$ is called {\it completely reducible}. Such a curve is called {\it reduced} if $m_i=1$ for each $i$.  We are interested in connected pencils of plane curves without fixed components and at least three completely
reducible fibers. By connectivity here we mean the nonexistence of a reduced fiber whose components intersect only at $\X$. For brevity we say that such a pencil is {\it of Ceva type}.

\begin{definition} The union of all completely reducible fibers (with a fixed partition into fibers, also called {\it blocks}) 
of a Ceva pencil of degree $d$ is called a  
($k,d$)-multinet where $k$ is the number of the blocks.  The base $\X$ of the pencil is determined by the multinet structure and called the base of the multinet.

 If the intersection of two fibers is transversal, i.e., $|\X|=d^2$ and hence all fibers are reduced then the multinet is called a net. If $|\X|<d^2$, i.e., some elements of $\X$ are provided with multiplicities
$m(P)>1$ but all completely reducible fibers are reduced then we call the multinet proper and light. Otherwise we call it heavy. 
 \end{definition}

From the viewpoint of projective geometry, a $(k,d)$-multinet is a multi-arrangement  $\mathcal{A}$ of lines in $\mathbb P^2$ provided with multiplicities $m(\ell)\in \mathbb{Z}_{>0}$ ($\ell\in \mathcal{A}$) and
partitioned into $k$ blocks $\mathcal{A}_1,\ldots, \mathcal{A}_k$ ($k\geq 3$) subject to the following two condition. 

(i) Let $\X$ be the set of the intersections of lines from different blocks. For each $P\in\X$, the number $$m(P)=\sum_{\ell\in \A_i, P\in \ell} m(\ell)$$ is independent on $i$. This number is called the {\it multiplicity} of $P$.

(ii) For every two lines $\ell$ and $\ell'$ from the same block there exists a sequence of lines from that block $\ell=\ell_0,\ell_1,\ldots,\ell_r=\ell'$ such that $\ell_{i-1}\cap\ell_i\not\in\X$ for $1\leq i\leq r$.

Thus, multinets can be defined purely combinatorially using an incidence relation. Note that  the multiplicity $m(\ell)$  for each $\ell\in \mathcal{A}$ equals the multiplicity of its corresponding linear factor in the completely reducible fibers of the Ceva pencil.

\subsection{Properties of multinets and examples that have been known}

There are several important properties of multinets. 

\begin{prop} 
Let $\mathcal{A}$ be a $(k,d)$-multinet. Then:\\

(1) $\sum_{\ell\in \mathcal{A}_i}m(\ell)=d$, independent of $i$;

(2) $\sum_{\ell\in \mathcal{A}}m(\ell)=dk$;

(3) $\sum_{P\in\X}m(P)^2=d^2$ (B\'ezout's theorem);

(4) $\sum_{P\in\X\cap\ell}m(P)=d$ for every $\ell\in \A$;

(5) There are no multinets with $k\geq 5$; 

(6) All multinets with $k=4$ are nets.

\end{prop}

The first four numerical equalities are easy and proved in \cite{FY}.  The last two properties are harder to establish and are proved in \cite{St, Yu2}. We now recall several examples of proper multinets that have been known since \cite{FY}. 

\begin{example}
A  $(k,1)$-net consists of $k$ lines intersecting all at one point with each block consisting of one point. This case corresponds to a so-called local resonance component. It is considered to be trivial and we will
often tacitly assume that $d>1$.
\end{example}

\begin{example} For each $n\geq1$, a $(3,2n)$-multinet is given by the pencil generated by polynomials
$x^n(y^n-z^n)$
and $y^n(x^n-z^n)$ with the third completely reducible fiber being $z^n(x^n-y^n)$. These are the projectivizations of the reflection arrangements for the full monomial groups $G(n,1,3)$ (see \cite{OT}). For $n=1$, it gives the only (up to projective isomorphism) $(3,2)$-net of Coxeter type $A_3$;
for $n=2$, it is the $(3,4)$-multinet of Coxeter type $B_3$. These multinets are heavy when $n>1$.
\end{example}

\begin{example}
The cubics $xyz$ and $x^3+y^3+z^3$ generate a Ceva pencil  with 4 completely 
reducible fibers. They give  the $(4,3)$-net known as the Hasse configuration.
It is the only known multinet with 4 blocks. A long-standing conjecture is that the Hasse configuration is the unique 4-net. 
\end{example}

\subsection{Constructions of nets}

From combinatorial viewpoint, $(k,d)$-nets are realizations of  $k-2$ pairwise orthogonal Latin squares of size $d$  (after identifying all blocks). 
 If $k=3$, the Latin square gives a multiplication table of a quasi-group. Thus one can view such a net as a representation of a Latin square or a quasi-group. (See \cite{Ynet}.)

If this quasi-group is a group, the representation can be reconstructed
using the complex torus $(\mathbb C^*)^2$. The list of groups that represent a net have
been recently completed by Korchmaros, Nagy and Pace in \cite{KNP1,KNP}
confirming a conjecture by Yuzvinsky in \cite{Yu3}. They also discovered new ways to construct the respective nets. Examples of nets representing quasi-groups which are not groups were constructed first by Stipins in \cite{St1}.

\section{Construction of Multinets}

\subsection{Multinets in higher dimensions}

It is possible to generalize the notion of multinet to ${\mathbb P^r}$ ($r>2$) using pencils of homogeneous polynomials of $r+1$ variables instead of 3. It is known that no multinet exists for $r\geq 5$ and  every multinet in ${\mathbb P^3}$ or ${\mathbb P^4}$ would be a net with 3 blocks (see \cite{PY}).

The only known nets in ${\mathbb P^r}$ for $r>2$ are the $(3,2n)$-nets ($n=1,2,\ldots$) in ${\mathbb P^3}$ given for every $n$ by the defining polynomial
$$ Q_n=[(x_0^n-x_1^n)(x_2^n-x_3^n)][(x_0^n-x_2^n)(x_1^n-x_3^n)][x_0^n-x_3^n)
(x_1^n-x_2^n)]$$
where the brackets determine the blocks.

This is the collection of all (projectivizations of) reflection hyperplanes of the finite complex reflection group known as the monomial group $G(n,n,4)$ (see \cite{OT}). For $n=2$ it is the Coxeter group of type $D_4$.

Each block of $Q_n$ is partitioned in two {\it half-blocks}
of degree $n$ each. Notice that all the planes of a half-block
intersect at one line, called the {\it base of the half-block}. For instance the base of the leftmost half-block is given by the system $x_0=0,x_1=0$.





\subsection{Construction of multinets} 

Unlike for nets, there have been no known systematic ways to construct proper multinets.
Here we suggest a way that has produced a variety of new examples.

Intersect  $Q_n$ with a plane $H$ that does not belong to $Q_n$. The resulting multi-arrangement in $H$ is denoted by $\A^H$  and referred to as the {\it arrangement induced by $Q_n$}. The pencil in $\mathbb P^3$ 
corresponding to $Q_n$ induces a pencil in ${\mathbb P^2}$ with 3 completely reducible fibers. It may happen that  the pencil has a fixed component
In this case we cancel the fixed components obtaining a smaller arrangement $\A^H_0$ with a multinet structure. Abusing the notation slightly we will call $\A^H$
 (if there is no fixed component) or $\A^H_0$, provided with the partitions
 into fibers of the induced pencil, the {\it induced multinet}.
 
 In the rest of the paper, the following agreement is applied. We use a homogeneous coordinate system $[x_0:x_1:x_2:x_3]$ in $\mathbb P^3$. If needed we can change the coordinates using symmetries of $Q_n$. In particular, in examples below we can always assume that the plane $H$ does not contain the point $[1:0:0:0]$ (by permuting coordinates if needed). Then intersecting with $H$ amounts to substituting $x_0$ by a linear combination of other coordinates from the equation of $H$. We also can change the coordinates multiplying them by any $n$th root of unity. We will also normalize 
 an equation for $H$ or homogeneous coordinates of a point dividing it by a non-zero number.

\section{Examples of multinets induced by $Q_n$}

\subsection{General position}
 
If $H$  does not contain any elements of the intersection lattice of $Q_n$ then $\A^H $ is a $(3,2n)$-net
 realizing the dihedral group of order $2n$. Moreover, $\A^H $ contains the $(3,n)$-net realizing the cyclic group $\mathbb{Z}_n$ as a subarrangement.
 For instance for $n=3$, $\A^H$ realizes the dihedral group of order 6 and contains the net realizing $\mathbb{Z}_3$ as a subarrangement. 
 
 \subsection{Heavy induced multinets (lines of multiplicity $n$)}
\label{heavy-n}
 
 If $H$ contains a line $\ell$ from the lattice of $Q_n$ then one of the following three situations happens. If $\ell$ is the base of a half-block then it becomes of multiplicity $n$ in $\A^H$.  If $\ell$ is the intersection of two planes from different half-blocks of a block
 then $\A^H$ gets a line of multiplicity  2. Finally if these two planes come from
 different blocks $\A^H$ gets a fixed component and cancellation is required.
 
 If $H$ does not contain a line from the lattice of $Q_n$ then $\A^H$ is light.
 
If $H$ contains the base of precisely one half-block then $\A^H$ is a proper heavy multinet with only one line of multiplicity $n$.
For instance, if $H$ is given by $x_0=cx_1$ ($c\not=0$ or is not a root of unity of degree $n$) then the blocks of $\A^H$ are determined by the polynomials: 
$$x_1^n(x_2^n-x_3^n), (c^nx_1^n-x_2^n)(x_2^n-x_3^n), (c^nx_1^n-x_3^n)(x_1^n-x_2^n)
.$$

If $H$ contains the bases of 2 half-blocks then
these half-blocks are from different blocks (since the bases of the half-blocks from the same block do not intersect in $\mathbb P^3$) and $H$ contains also the base
of a half-block from the third block. Such $H$ is a coordinate plane, for example, 
$x_0=0$. Then every block of $\A^H$ contains exactly one line of multiplicity $n$ and the underlying arrangement is 
 the reflection arrangements for the full monomial groups $G(n,1,3)$ (see above). For instance, if $H$ is given by $x_0=0$ then $\A^H$ 
has blocks: $$x_1^n(x_2^n-x_3^n), x_2^n(x_1^n-x_3^n), x_3^n(x_1^n-x_2^n).$$

\subsection{Heavy induced multinets (lines of multiplicity 2)}
\label{heavy-2}
$\A^H$ has a line of multiplicity 2 if and only if  $H$ contains the line of intersection
of precisely two planes of $Q_n$ from different half-blocks of the same block. If $H$ is generic with that condition then there are no fixed components and $\A^H$ has precisely one  line of multiplicity 2. Also $H$ could contain several such lines from different blocks which produces
up to three lines of multiplicity 2  if $n$ is even, and up to two such lines
if $n$ is odd. 
For instance, if $H$ is given by $a(x_0-x_1)-(x_2-x_3)=0$ with $a$ not 0 or a root of unity of degree $n$ then $\A^H$ has the factor $(x_2-x_3)^2$ in one block and all other factors have multiplicity 1. If $a=1$ and $n$ is odd,  $\A^H$ also contains the factor $(x_1-x_3)^2$ in another block. If $a=1$ and $n$ is even, $\A^H$ contains additionally the factor $(x_1+x_2)^2$ in the third block. The multiplicities of all other factors are equal to 1 in each of these situations.
A similar effect is produced by $a$ equal to a root of unity of degree $n$.

Finally it is easy to see that $\A^H$ cannot have both: a line of multiplicity $n$ and a line of multiplicity 2. Indeed by subsection \ref{heavy-n} the former forces 
$H$ to have equation of the form $Ax_i-Bx_j=0$ for some distinct $i$ and $j$.
If now $H$ contains the intersection of two planes from a block of $Q_n$ then this block must coincide with the block containing $x_i^n-x_j^n$. Furthermore
$H$ must coincide with  a plane of this block (given by $x_i-\zeta x_j=0$
where $\zeta$ is a root of unity) which is forbidden.

\subsection{Light induced multinets (points of multiplicity $n$)}
\label{light-n}
In the rest of this section, by a `point' and `line' we mean respectively a point or a line
of $\mathbb P^3$ from the intersection lattice of  $Q_n$.  
If besides any of those lies in $H$ then they have multiplicity coming from
multinet $\A^H$ or $\A^H_0$.

Recall that a multinet is light if it is proper (i.e., not a net) and all of its lines have multiplicity 1. In order to produce a light multinet with a point of multiplicity $n$ the plane $H$ must contain the point of intersection of two bases of half-blocks (they must be from different blocks). 

We can make this more concrete. Every one of six bases of half-bloks can be given by the equations $x_i=x_j=0$ where $\{i,j\}\subset\{1,\ldots,6\}$. Also two non-disjoint bases intersect at one of the points having one coordinate 1 while others 0.  Thus the induced multinet is light with precisely one point of multiplicity $n$
(and others of multiplicity 1) if and only if 
$H$ contains precisely one of these points and is generic otherwise. For instance, if $H$ is given by $Ax_0+Bx_1+Cx_2=0$ with otherwise generic coefficients, then
the induced multinet is light with the only one multiple point $[0:0:0:1]$ of multiplicity
$n$ (in $\A^H$). 

Light $\A^H$ can have more than one point of multiplicity $n$. For instance, $H$
given by $Ax_0+Bx_1=0$ has both $[0:0:0:1]$ and $[0:0:1:0]$. It cannot though
have three such points. Indeed if $AB=0$ then $\A^H$ is heavy (see subsection \ref{heavy-n}).

\subsection{Light induced multinets (points of multiplicity 2)}
When we discuss points of multiplicity 2 of $\A^H$ we always assume that $n>3$ 
in order not to confuse them with points of multiplicity $n$ and (for the presence of fixed components) $n-1$.

Since a point $P$ of multiplicity 2 has exactly 2 lines from each block intersecting at it the plane $H$ must contain the point of 
intersection of six planes, two from each block of $Q_n$. Moreover planes in  every one of these pairs must be from different
half-blocks since otherwise $H$ would contain the base of a block and the induced multinet would be heavy. Conversely if $H$ contains the intersection of such planes then 
$\A^H$ has a point of multiplicity 2.
If $H$ is generic otherwise then $\A^H$ is light and contains precisely one point of multiplicity 2.

For instance, if $H$ is given by $Ax_0+Bx_1+Cx_2+Dx_3=0$ with $A+B+C+D=0$
(whence passing through $P=[1:1:1:1]$) and generic otherwise then  $P$ has multiplicity 2 in $\A^H$ while other lines and points have multiplicity 1.

Let us notice that for a point $P'$ of intersection of only four planes, 2 from one block and 2 from another, there are 2 planes from the third block passing through $P'$.

For the future use we can characterize more explicitly points of those intersections. 
Let 4 planes be given by 
$$x_0-\zeta^a x_1=0,\ x_2-\zeta^b x_3=0,\  x_0-\zeta^c x_2,\ x_1-\zeta^d x_3$$
where $\zeta$ is  primitive root of 1 and its exponents are arbitrary from a cyclic group
$C_n$ (in additive notation). Then for the intersection to exist the equality $a+d=b+c$ is needed
and then the intersection is 
$$[\zeta^{a+d}:\zeta^d:\zeta^b:1]$$. 

Conversely, let $P'\in H$ where $P'=[\zeta^a:\zeta^b:\zeta^c:1]$ with $\zeta,a,b,c$ are as above. Clearly $P'$ lies in the four planes:
$$x_0-\zeta^ax_3=0,\ x_1-\zeta^bx_3=0,\ x_1-\zeta^{b-c}x_2, \ x_0-\zeta^{a-c}x_2$$
that proves the converse.

In particular this proves the following Lemma.

\begin{lemma}
\label{roots}
Any point $P$ of $\mathbb P^3$  has multiplicity 2 in $\A^H$ for every allowable 
plane $H$ passing through it if and only if it  has homogeneous coordinates that are roots of unity of degree $n$.
\end{lemma}

\subsection{Light induced multinets (several points of multiplicity 2)}
A plane $H$ can have several points described in the previous subsection whence 
$\A^H$ can have several points of multiplicity 2. A partial classification of induced light multinets with double points  for $n\leq 6$ 
is given in \cite{B}. 

First we consider light induced  multinets without points of multiplicity $n$.
The current maximal number of points of multiplicity 2 known for examples of light multinets is 8.

\begin{example} Take $n=8$ and fix a primitive 8th root of unity $\zeta$. Let $H$ be given by
 
$$x_0-(\zeta+1)x_1-\zeta^3x_2+(\zeta^3+\zeta)x_3=0.$$ 

Then $\A^H$  is light,  has no fixed components and has 8 points of multiplicity 2 (with all other points having multiplicity 1). These 8 points are as follows:

$$\begin{array}{ll}
[1:1:1:1] \hspace{1cm} \phantom{}& [\xi^5:\xi^2:\xi^3:1] \\
\phantom{}[\xi^2:\xi:1:1] & [\xi^5:\xi^3:\xi^5:1] \\
\phantom{}[\xi^2:\xi^2:\xi^6:1] & [\xi^7:1:\xi:1] \\
\phantom{}[\xi^4:\xi^3:\xi^6:1] & [\xi^7:\xi:\xi^3:1]. \\
\end{array}$$\\

\noindent Each of these points lies on $H$ and on exactly six hyperplanes of $Q_8$ (one from each half-block). For instance, $[\xi^2:\xi:1:1]$ lies on the six hyperplanes $$\begin{array}{lll}
x_0-\xi x_1 \hspace{0.5cm} \phantom{}& x_0-\xi^2 x_2 \hspace{0.5cm} \phantom{}& x_0-\xi^2 x_3 \\
x_2-x_3 & x_1-\xi x_3 & x_1-\xi x_2. \end{array}$$\\

\end{example}

On the other hand, the following surprising result holds.

\begin{thm}
\label{no n}
Let $n>3$ and $\A^H$ a light induced multinet without points of multiplicity $n$. Then the number of points of multiplicity 2 in it 
 is less than $2^{96}$ (independently of $n$).
\end{thm}

\begin{proof}
We fix $H$ and without any loss assume that it is given by $Ax_0+Bx_1+Cx_2-x_3=0$.
If $\A^H$ has a point of multiplicity 2 then by Lemma \ref{roots}
\begin{equation}
\label{eq:two}
A\zeta^a+B\zeta^b+C\zeta^c=1
\end{equation}
for some primitive $n$-th root of unity $\zeta$ and $a,b,c\in C_n$. 

Now we prove that relation (\ref{eq:two}) of roots of unity is non-degenerate meaning that no proper partial sum of the lefthand side is 0. If for instance $A=0$ then $H$ contains the point $[1:0:0:0]$ that has multiplicity $n$ in $\A^H$  which contradicts a condition of the theorem.

Moreover suppose  $A\zeta^a+B\zeta^b=0$, i.e., $[A:B]=[\zeta^b:-\zeta^a]$
and $C=\zeta^{-c}$. Then the equation for $H$ becomes
$$D\zeta^bx_0-D\zeta^ax_1+\zeta^{-c}x_2-x_3=D\zeta^b(x_0-\zeta^{a-b}x_1)+
\zeta^{-c}(x_2-\zeta^cx_3)=0.$$
This form of the equation shows that $H$ contains the line of intersection
of 2 planes (from one block): $x_0-\zeta^{b-a}x_1=0$ and $x_2-\zeta^cx_3=0.$
The intersection of these planes with $H$ gives in $\A^H$ a line of multiplicity 2 whence $\A^H$ is heavy.

We finish the proof applying the result from \cite{roots} (for $k=3$) which says that the number of non-degenerate solutions in roots of unity of a given equation 
$\sum_{i=1}^kA_ix_i=1$ with complex coefficients is bounded from above by 
$2^{4(k+1)!}$.

\end{proof}

\subsection{Fixed components}

First we rephrase the existence of fixed components in terms of multiplicity of points.
\begin{lemma}
\label{fixed}
Assume that $\A^H$ does not have lines of multiplicity greater than 1. Then
$\A^H$ has both a point of multiplicity $n$ and a point of multiplicity 2 if and only if
$\A^H$ has a fixed component. Besides if $\A^H$ has points of multiplicity 2 it cannot have more than one point of multiplicity $n$.
\end{lemma}
\begin{proof}
Using the conditions on $H$ we can assume without any loss that (i) $[0:0:0:1]\in H$
and (ii) there is a point $[\zeta^a:\zeta^b:1:\zeta^c]$ in $H$
 where $\zeta$ is a primitive root of unity of degree $n$
while $a,b,c\in C_n$ (see Lemma \ref{roots}).

Thus
$H$ can be given by an equation of the form
$$Ax_0+Bx_1+Dx_2=0$$ 
with $D=-A\zeta^a-B\zeta^b$. Plugging it into the equation we can write the latter as
\begin{equation}
\label{eq:fixed}
A(x_0-\zeta^ax_2)=-B(x_1-\zeta^bx_2).
\end{equation}
In the equation $A,B\not=0$. Indeed if only one of them is 0 then $H$ coincides with a plane of $Q_n$ which is forbidden. If both vanish then $H$ is given by $x_2=0$
which gives a line in $\A^H$ of multiplicity $n$.
This implies that two blocks of $\A^H$ have $x_1-\zeta^ax_3$ as a 
common component whence all three of them do. 

In order to prove the converse suppose $\A^H$ has a fixed component.
This is equivalent to $H$ containing the line of intersection of two planes
from different block.
Thus without any loss we can assume that $H$ is given by  an equation of the type
\eqref{eq:fixed} which implies the statement.

If  $\A^H$ had two points of multiplicity $n$ then without any loss we
could assume that $B=0$ above which is a contradiction as in the first part of the proof.

\end{proof}

\begin{remark}
\label{two-fixed}
The proof of the Lemma \ref{fixed} really gives a fixed components
passing through the given points of multiplicity $n$ and 2. Conversely, every
fixed component has a unique point of multiplicity $n$ and $n$ points of 
multiplicity  2. For instance, if coordinates are chosen as in the proof of the Lemma \ref{fixed}
one obtains $n$ points of multiplicity 2 keeping the coordinates $x_0,x_1,x_2$ 
satisfying \ref{eq:fixed} and taking $x_3=\zeta^d$ $(d\in C_n)$.
\end{remark}

Now we assume that $\A^H$ has a fixed component whence $H$ has both,
a point of multiplicity $n$ and a point of multiplicity 2.
The general form of an equation of such $H$ is (up to permuting coordinates)
$$A(x_0-\zeta^ax_1)+B(x_0-\zeta^bx_2)=0$$
with $AB(A+B)\not=0$, $\zeta$ is a primitive root of unity and $a,b\in C_n$. 
Notice that $P_0=[0:0:0:1]\in H$. After substitution $\frac{A\zeta^ax_1+B\zeta^bx_2}{A+B}$ for $x_0$ and canceling $x_1-\zeta^{b-a}x_2$ we have 
a multinet $\A^H_0$ such that $P_0$ acquires multiplicity $n-1$. 

This point may be the unique point of multiplicity $n-1$. For instance, this the case of $H$ given by $3x_0-2x_1-x_2=0$.

Choosing $H$ more carefully we can have it containing besides the intersection of 
another pair of planes from the same two blocks. For instance, if $H$ is given by $x_0=(\zeta+1)x_1-\zeta x_2$  then  it produces the common factors $x_1-x_2$ and $x_1-\zeta x_2$ which results in a light $(3,2n-2)$-multinet upon cancellation.
The point $P_0$ again becomes the only point of multiplicity $n-2$ of 
$\A^H_0$.

According to Lemma 2, every $H$ such that $\A^H$ has a fixed component 
must have not only a point of multiplicity $n$ but also some points of multiplicity 2
that all lies on fixed components. Thus  $\A^H_0$ does not have any points of multiplicity 2.

\begin{remark}
At the beginning of Lemma \ref{fixed} we assume that $\A^H$ does not any 
multiple points. Using the technique of this subsection it is almost immediate
to prove that a heavy induced multinet cannot have fixed components.
\end{remark}

\subsection{The number of mixed components}
The question that was not addressed so far is how many mixed components 
the induced arrangement $\A^H$ might have. We have seen above that it may be 2. Now we resolve the question.

\begin{thm}
\label{number-mixed}
No induced arrangements $\A^H$ can have more than 2 mixed components.
\end{thm}
\begin{proof}
Suppose $\A^H$ has a mixed component. By results of the previous subsection $H$ can be given by 
$$(x_0-x_1)+A(x_1-x_2)=0$$
after normalizing and changing the coordinates if needed.
Then any other mixed component would force the lefthand side of this equation to be equal to
$$(x_0-\zeta^ax_1)+A(\zeta^bx_1-x_2)$$
for a primitive root of unity $\zeta$ and some $a,b\in C_n$.
The complex  numbers $\zeta^a,\zeta^b$ and $A$ are related by the following formula:

\begin{equation}
\label{eq:relation}
A=\frac{1-\zeta^a}{1-\zeta^b}
\end{equation}
where $a,b\not=0$.

Thus in order to prove the theorem it suffices to prove  that the relation  (3)
can have at most one solution for fixed $A$ with $a,b\in C_n\setminus\{0\}$. Notice also that $a\not=b$ since if $a=b$ then $A=1$
whence $H$ coincides with a plane from $Q_n$ which is forbidden. 

In fact we prove a stronger statement as the following.

\begin {prop} Let $A\in \C$. The equation \begin{equation}\tag{4}A=\frac{1-\xi}{1-\eta}\end{equation} has at most one solution $(\xi, \eta)\in \C^2$ with the properties: $\xi\ne \eta$,  $|\xi|=|\eta|=1$, and $\xi, \eta \ne 1$. \\
\end{prop}

\begin{proof} Suppose that there are two solutions of $(4)$: $(\xi,\eta)$ and $(\xi_1,\eta_1)$ and their respective complex arguments are $\alpha, \beta, \alpha_1,\beta_1 \in (0,2\pi)$. It is clear that for any number $z=e^{i\gamma}$ we have $$1-z=2\sin\left(\frac{\gamma}{2}\right)e^{i\left(\frac{\gamma-\pi}{2}\right)}.$$
Hence equation (4) implies 
\begin{equation}\tag{5} \alpha-\beta=\alpha_1-\beta_1\end{equation}  \begin{equation}\tag{6} \sin \left(\frac{\alpha}{2} \right)\sin \left(\frac{\beta_1}{2} \right)=\sin \left(\frac{\alpha_1}{2} \right)\sin \left(\frac{\beta}{2} \right) \end{equation}
\begin{equation}\tag{7}\cos \left(\frac{\alpha}{2} \right)\cos \left(\frac{\beta_1}{2} \right)=\cos \left(\frac{\alpha_1}{2} \right)\cos \left(\frac{\beta}{2} \right)\end{equation} where equation (5) is the equality of the arguments, equation (6) comes from the equality of the moduli, and equation (7) follows from (5,6). Also, $(5,6,7)$ imply the following equations for cotangents \begin{equation}\tag{8} \cot\left(\frac{\alpha}{2}\right) \cot\left(\frac{\beta_1}{2}\right)= \cot\left(\frac{\alpha_1}{2}\right) \cot\left(\frac{\beta}{2}\right) \end{equation} \begin{equation}\tag{9} \cot\left(\frac{\alpha}{2}\right)- \cot\left(\frac{\beta}{2}\right)= \cot\left(\frac{\alpha_1}{2}\right)- \cot\left(\frac{\beta_1}{2}\right). \end{equation}  \indent If $\cot\left(\frac{\beta}{2}\right)=0$, i.e., $\frac{\beta}{2}=\frac{\pi}{2}$, then (8) implies either $\cot\left(\frac{\alpha}{2}\right)=0$ whence $\xi=\eta$ or $\cot\left(\frac{\beta_1}{2}\right)=0$ whence $\eta=\eta_1$. The first situation is not possible by the hypotheses of the proposition. The second and (4) show $(\xi_1,\eta_1)=(\xi,\eta)$.\\

 If $\cot\left(\frac{\beta}{2}\right)\ne 0$, resolve equation (8) for $\cot\left(\frac{\alpha_1}{2}\right)$ and plug it into equation (9). Since $\frac{\alpha}{2},\frac{\beta}{2}\in(0,\pi)$ and $\alpha\ne\beta$, we can cancel the factor $\cot\left(\frac{\alpha}{2}\right)-\cot\left(\frac{\beta}{2}\right)$to obtain $\cot\left(\frac{\beta_1}{2}\right)=\cot\left(\frac{\beta}{2}\right)$ whence $\frac{\beta_1}{2}=\frac{\beta}{2}$. The equality $(\xi_1,\eta_1)=(\xi,\eta)$ again follows.  
\end{proof}

The statement of the theorem is a particular case of the proposition.
\end{proof}

\begin{remark}
\label{new-proof}
When the first version of the paper appeared on arXiv we received a comment from Joe Buhler with a more elegant 
proof (of the previous theorem) whose main idea he attributed to Richard Stong. With their permission, we exhibit this proof
below.

When $z$ ranges over the pointed unit circle without $1$ in the complex plane the set of $1-z$ is the pointed unit circle $C$ centered at 1 without 0.  The proposition is about the multiplicative relation of the form

                    $$ ab = cd$$

\noindent where $a,b,c,d\in C$, and (say) $a$ is distinct from both $c$ and $d$.  By taking inverses, this is equivalent to the same relation in the set $D = C^{(-1)}$ of inverses of elements of $C$, which is the set of complex numbers of real part $\frac{1}{2}$.  But if

          $$ (1/2 + ir)(1/2+is) = (1/2+it)(1/2+iu)$$         

\noindent with $r,s,t,u$ real
then equating real and imaginary parts one deduces that the sums and products of sets $\{r,s\}$ and $\{t,u\}$ are equal. Thus these sets coinside.
\end{remark}

 \subsection{Summary of properties of induced multinets from $Q_n$}
 The multinets induced from $Q_n$ possess the following properties 
(we suppose $n>3$).

\medskip

1. The multiplicity of lines takes only values 1, 2, and $n$. There can be 1 or 3 lines of multiplicity $n$. There can be one line of multiplicity 2. Also, there can be three such lines (if $n$ is even) or two (if $n$ is odd).

\bigskip

2. The multiplicity of points takes values from the  list $\{1,2,n-2,n-1,n\}$.
\bigskip

3. A light multinet can have up to two points of multiplicity $n$ or at most one point
of multiplicity $n-1$ or at most one point of multiplicity $n-2$. These three cases are disjoint and each does not allow any other point with multiplicity larger than 1.
\bigskip

4. A light multinet can have several points of multiplicity 2 if it does not have  points of multiplicity $n$. The number of these points is bounded independently of $n$ by
$2^{96}$. 

\section{Combinatorics inside blocks}

In this section we discuss the possibilities for combinatorics of lines and points
inside a block of an induced multinet.

First suppose that $\A^H$ is light, does not have a fixed component and $H$ is given by the equation $Ax_0+Bx_1+Cx_2+Dx_3=0$. Because of the condition on $\A^H $
the plane $H$ does not contain any base line. On the other hand, it has
one point of intersection with every base line. For instance, for the half-block 
$x_0^n-x_1^n$ the point is $[0:0:-D:C]$. Thus in $\A^H$ every
half-block consists of lines intersecting all at one point (i.e., forming a pencil
of dimension 1). For generic $H$ the base point of the pencil is not in $\X$
whence these points for different half-blocks are distinct.  If however
$H$ is passing through the intersection of two base lines then that point is in $\X$. For instance if $H$ is given by $Ax_0+Bx_1+Cx_2=0$ then the point $[0:0:1]$ in it 
is the base point of the pencils in 3 half-blocks from 3 different blocks. 

If $\A^H$ has fixed components then the same claim holds
for $\A^H_0$ except the amount of lines in each of 3 half-blocks involved
in the fixed components decreases either to $n-1$ or $n-2$.

Now suppose that $\A^H$ is heavy. If it has lines of multiplicity $n$ (one or three)
then the respective half-block contains precisely one of these lines instead of a pencil.
If $\A^H$ has lines of multiplicity 2 then one of the lines in the respective block 
connect the  base vertices of two half-blocks (i.e., the intersections of the half-block bases with $H$). Recall that $\A^H$ cannot have lines of both multiplicities: $n$ and 2 (see \ref{heavy-2}).

Summing up the discussion in this section we conclude that the multiplicities of lines 
of an induced multinet determine the combinatorics inside blocks.  

\section{Conjectures and open problems}

For multinets there are more open questions than answers. Here are some of the former.

{\bf Problem 1.} To make the upper bound in Theorem \ref{no n} smaller. (We conjecture that it can be significantly decreased.) \\

{\bf Problem 2.} All induced multinets can be obtained from nets by deformation
(moving the plane $H$). Prove the conjecture from \cite{PY} that all multinets have this property. \\

{\bf Problem 3.}  Are there nets in $\mathbb P^3$ other than $Q_n$? \\

{\bf Problem 4.}  There are nets which are not induced by $Q_n$. Such as, for example, every $(3,2k+1)$-net for $k=1,2,\ldots$ . The light multinet in Figure 2
of \cite{FY} is also not induced from $Q_n$ (a proof should include that it is not induced from
$Q_6$ after a cancellation). 

Are there heavy multinets not induced from $Q_n$? \\

\end{document}